\documentclass[final,onefignum,onetabnum]{siamart171218}

\usepackage[
  margin=3.7cm,
  includefoot,
  footskip=30pt,
]{geometry}
\usepackage{layout}

\usepackage{mathtools}

\usepackage{amstext,amssymb,amsmath}
\usepackage{relsize,xspace,ifthen}
\usepackage[T1]{fontenc}
\usepackage{lmodern}
\usepackage{microtype}
\usepackage{enumitem} 
\usepackage[utf8]{inputenc}
\usepackage{tikz}
\usetikzlibrary{arrows}
\usetikzlibrary{shapes.geometric}
\usetikzlibrary{calc, decorations.markings}
\usepackage{comment}
\usepackage[hyperref]{ntheorem}

\newenvironment{cenv}{\begin{list}{}{%
      \setlength{\labelwidth}{1.5em}%
      \setlength{\leftmargin}{\labelwidth}%
      \addtolength{\leftmargin}{\labelsep}%
      \setlength{\listparindent}{0em}%
      \setlength{\topsep}{10pt}%
      \setlength{\itemsep}{5pt}%
      \setlength{\parsep}{0pt}%
    }
  }{
  \end{list}
}

\newcounter{claimcounter}

\newenvironment{Claim}{
  
  \refstepcounter{claimcounter}
  \begin{cenv}
  \item[{Claim \arabic{claimcounter}.}]
  }{
  \end{cenv}
}

\usepackage{latexsym}
\usepackage{nicefrac}

\usepackage{algorithm}
\usepackage{algpseudocode}
\usepackage{eqparbox}

\newtheorem{rmrk}[theorem]{Remark}
\newcommand{\minor}{\preceq}
\newcommand{\N}{{\mathbb N}}

\newcommand{\CCC}{\mathcal{C}}

 \newcommand{\XXX}{\mathcal{X}}
\newcommand{\YYY}{\mathcal{Y}}

\newcommand{\Oof}{\mbox{$\cal O$}}
\newcommand{\rad}{\mathrm{rad}}

\newcommand{\wcol}{\mathrm{wcol}}

\newcommand{\Wreach}{\mathrm{WReach}}
\newcommand{\Sreach}{\mathrm{SReach}}

\newcommand{\col}{\mathrm{col}}
\newcommand{\adm}{\mathrm{adm}}

\newcommand{\dist}{\operatorname{dist}}

\newcommand{\NOdM}{Ne{\v s}et{\v r}il and Ossona de Mendez}

\renewcommand{\phi}{\varphi}
\renewcommand{\epsilon}{\varepsilon}

\newcommand{\ie}{i.e.\@\xspace}
\newcommand{\WLOG}{w.l.o.g.\@\xspace}
\renewcommand{\mid}{~:\ }

\newcounter{rbcounter}
\DeclareMathOperator{\tw}{tw}
\DeclareMathOperator{\td}{td}

\usepackage[absolute]{textpos}

\title{Colouring and Covering Nowhere Dense Graphs\thanks{A conference version of 
this paper was published at the 41st International Workshop on Graph-Theoretic Concepts in Computer Science, WG 2015. The jounal version contains new additional results. 
Stephan Kreutzer and Roman Rabinovich's 
research has been supported by the
European Research Council (ERC) under the European Union’s Horizon
2020 research and innovation programme (ERC Consolidator Grant
DISTRUCT, grant agreement No 648527). The work of Sebastian Siebertz is supported by the National Science Centre of Poland via POLONEZ grant agreement UMO-2015/19/P/ST6/03998, 
which has received funding from the European Union's Horizon 2020 research and 
innovation programme (Marie Sk\l odowska-Curie grant agreement No.\ 665778).}}

\author{Martin Grohe\thanks{RWTH Aachen University, \texttt{grohe@informatik.rwth-aachen.de}}
\and  Stephan Kreutzer\thanks{Technische Universit\"at Berlin, \texttt{stephan.kreutzer@tu-berlin.de}} \and
Roman Rabinovich\thanks {Technische Universit\"at Berlin, \texttt{roman.rabinovich@tu-berlin.de}}\and \newline
Sebastian Siebertz\thanks{University of Warsaw, \texttt{siebertz@mimuw.edu.pl}}\and
Konstantinos Stavropoulos\thanks{Universit\"at Hamburg, \texttt{konstantinos.stavropoulos@uni-hamburg.de}}}

\begin{document}
\maketitle
\begin{textblock}{5}(12.05, 12.2)
\includegraphics[width=30px]{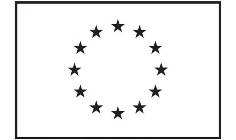}%
\end{textblock}

\begin{abstract}
  In \cite{GKS14} it was shown that nowhere dense classes of graphs
  admit sparse \emph{neighbourhood covers} of small degree. We
  show that a monotone graph class admits sparse neighbourhood covers
  if and only if it is nowhere dense.  The existence of such covers
  for nowhere dense classes is established through bounds on so-called
  \emph{weak colouring numbers.}
  The core results of this paper are various lower and upper bounds on
  the weak colouring numbers and other, closely related generalised
  colouring numbers. We prove tight bounds for these numbers on graphs
  of bounded treewidth. We clarify and tighten the relation between
  the density of shallow minors and the various generalised colouring
  numbers. These upper bounds are complemented by new, stronger
  exponential lower bounds on the weak and strong colouring numbers, 
  and by super-polynomial lower bounds on the weak colouring numbers on 
  classes of polynomial expansion. 
  Finally, we show that computing weak $r$-colouring numbers is
  NP-complete for all $r\geq 3$.
\end{abstract}

\section{Introduction}

Nowhere dense classes of graphs have been introduced by  \NOdM~\cite{NesetrilOdM12,NesetrilO11} as a general model
of \emph{uniformly sparse graph classes}. They include and generalise many other
natural  sparse graph classes, among them all classes of bounded degree,
classes of bounded genus, classes
defined by excluded (topological) minors, and classes of bounded
expansion. It has been demonstrated in several papers, e.g.,
\cite{DK09,EickmeyerGKKPRS17,GKS14,KreutzerRS17,NesetrilOdM12} that nowhere dense graph classes have
nice algorithmic properties; many problems that are hard in general
can be solved (more) efficiently on nowhere dense graph classes. In fact, nowhere dense classes are a natural limit for the
efficient solvability of a wide class of problems
\cite{DvorakKT13,GKS14,Kreutzer11}.

In \cite{GKS14}, it was shown that nowhere dense classes of graphs
  admit sparse \emph{neighbourhood
    covers}. Neighbourhood covers play an important role in the study
  of distributed network algorithms and other application areas (see
  for example~\cite{peleg00}). The
  neighbourhood covers developed in~\cite{GKS14} combine low radius
  and low degree making them interesting for the applications outlined
  above.  In this paper, we prove a (partial) converse to the result
  of \cite{GKS14}: we show that monotone graph classes (that is,
  classes closed under taking subgraphs) are nowhere dense if and only
  if they admit sparse neighbourhood covers. A similar characterisation result was 
  proved for classes of bounded expansion~\cite{NesetrilM15}.

Nowhere denseness has turned out to be a very robust property of graph
classes with various seemingly unrelated characterisations (see
\cite{GKS13,NesetrilOdM12}), among them characterisations through
so-called \emph{generalised colouring numbers}. These are particularly
relevant in the algorithmic context, because the existence of sparse
neighbourhood covers for nowhere dense classes is established through
such colouring numbers---the \emph{weak $r$-colouring numbers}, to be
precise---and the value of these numbers is directly related to the
degree of the neighbourhood covers. Besides the weak $r$-colouring numbers
$\wcol_r(G)$ of graphs $G$ we study the
\emph{$r$-colouring numbers} $\col_r(G)$ and the \emph{$r$-admissibility
  numbers} $\adm_r(G)$. The two families of colouring numbers where introduced by Kierstead
and Yang in~\cite{kierstead03}, and the admissibility numbers go back to
Kierstead and Trotter in~\cite{kierstead93} and were generalised by Dvo{\v{r}}{\'a}k
in~\cite{dvovrak13}. All these numbers generalise the
\emph{degeneracy}, a.k.a. \emph{colouring number}, which
is defined to be the minimum~$d$ such that there is a linear order
of the vertices of $G$ in which every vertex has at most~$d$ smaller
neighbours. The name ``colouring number'' comes from the fact that
graphs of degeneracy $d$ have a proper $d+1$ colouring which can be
computed efficiently by a simple greedy algorithm. For the generalised
$r$-colouring numbers, instead of smaller neighbours of a vertex we
count smaller vertices reachable by certain paths of length $r$; the
numbers differ by the kind of paths of length $r$ 
considered. We observe that with growing $r$ the colouring numbers
converge to the treewidth of the graph. 

The core results of this paper are various upper and lower bounds for these
families of colouring numbers.  In particular, we prove tight bounds
for $\wcol_r(G)$ for graphs~$G$ of bounded
treewidth. We clarify and tighten the relation between the density
of shallow minors and the
various generalised colouring numbers. 
These upper bounds are complemented by new, stronger
exponential lower bounds on the strong and weak colouring numbers. The lower
bounds can already be achieved on graph classes of bounded degree. 
We furthermore show that there exist classes of polynomial expansion on
which the weak colouring numbers grow super-polynomially in $r$. 
This result answers negatively a question of Joret and Wood, 
whether graph classes of polynomial
expansion have polynomial weak colouring numbers. 
Finally, we show that  computing weak $r$-colouring numbers
is NP-complete for all $r\geq 3$.

\smallskip
After giving some graph theoretic background in
Section~\ref{sec:prelim}, we prove our various bounds on the
generalised colouring numbers in
Sections~\ref{sec:colourings}--\ref{sec:reghighgirth}.
Section~\ref{sec:covers} is devoted to sparse neighbourhood covers,
and the NP-completeness result for the weak colouring numbers is
proved in Section~\ref{sec:NPcomplete}. 

\section{Generalised Colouring Numbers}
\label{sec:prelim}

Our notation from graph theory is standard, we refer the reader to \cite{Diestel} for background.  All graphs in this paper are finite and simple, \ie they do not have self-loops or multiple edges. A class of graphs is \emph{monotone} if it is closed under subgraphs. The \emph{radius}~$\rad(G)$ of~$G$ is $\min_{u\in V(G)}\max_{v\in V(G)}\dist^G(u,v)$.  By~$N_r^G(v)$ we denote the \emph{$r$-neighbourhood} of~$v$ in~$G$, \ie the set of vertices of distance at most~$r$ from~$v$ in~$G$. 


We represent a linear order on $V(G)$ as an injective function $L: V(G) \rightarrow \N$ and write~$\Pi(G)$ for the set of all linear orders on $V(G)$.

Vertex $u$ is \emph{weakly $r$-reachable} from $v$ with respect to the
order $L$, if there is a path $P$ of length $0\leq \ell\leq r$ from
$v$ to $u$ such that $L(u) \le L(w)$ for all $w\in V(P)$.  Let
$\Wreach_r[G,L, v]$ be the set of vertices that are weakly
$r$-reachable from $v$ with respect to~$L$.
If furthermore, all inner vertices $w$ of $P$ satisfy $L(v) < L(w)$, then $u$ is
called \emph{strongly $r$-reachable} from $v$. Let $\Sreach_r[G, L,
v]$ be the set of vertices that are strongly $r$-reachable from~$v$
with respect to $L$.

The \emph{$r$-admissibility} $\adm_r[G,L, v]$ of $v$ with respect to
$L$ is the maximum size~$k$ of a family $\{P_1,\ldots,P_k\}$ of paths
of length at most $r$ in $G$ that start in~$v$, end at a vertex $w$ with
$L(w) \le L(v)$ and satisfy $V(P_i)\cap V(P_j)=\{v\}$ for
$1\leq i\neq j\leq k$. As we can always let the paths end in the first
vertex smaller than~$v$, we can assume that the internal vertices of
the paths are larger than $v$. Note that $\adm_r[G,L,v]$ is an
integer, whereas $\Wreach_r[G,L, v]$ and $\Sreach_r[G,L,v]$ are sets
of vertices.

The \emph{weak
  $r$-colouring number} $\wcol_r(G)$, the \emph{$r$-colouring number}
$\mathrm{col}_r(G)$, and the \emph{$r$-admissibility}
$\mathrm{adm}_r(G)$ are defined as \begin{align*}
  \wcol_r(G)&=\min_{L\in\Pi(G)}\max_{v\in
    V(G)}|\Wreach_r[G,L, v]|,\\
  \mathrm{col}_r(G)&=\min_{L\in\Pi(G)}\max_{v\in
    V(G)}|\Sreach_r[G,L, v]|,\\
  \mathrm{adm}_r(G)&=\min_{L\in\Pi(G)}\max_{v\in
    V(G)}\adm_r[G,L,v].
\end{align*}

It follows from the definitions that, for all $r\in\N$,
$\adm_r(G)\leq \col_r(G)\leq \wcol_r(G)$. Also,
$\adm_1(G)\leq \adm_2(G)\leq \ldots\leq \adm_n(G)$,
$\wcol_1(G)\leq \wcol_2(G)\leq \ldots \leq \wcol_n(G)=\td(G)$ (where $\td(G)$ is the treedepth of~$G$, see
e.g.~\cite{NesetrilOdM12}) and
$\col_1(G)\leq \col_2(G)\leq \ldots\leq \col_n(G)=\tw(G)+1$ (where $\tw(G)$ is the treewidth of~$G$).

To see that $\col_n(G) = \tw(G)+1$, note that treewidth can be
characterised by elimination orders. An \emph{elimination order} of a
graph $G$ is a linear order $L$ on $V(G)$ with which we associate a
sequence of graphs $G_i$. Let $V(G) = \{1,\ldots,n\}$ and $L(i)<L(j)$
for $i<j$, then $G_0 = G$ and for $0<i\le n$,
$V(G_i) = V(G_{i-1})\setminus \{i\}$ and
\[E(G_i) = \Bigr(E(G_{i-1})\setminus \bigl\{\{i,j\} \mid j\le
n\bigr\}\Bigl) \cup \bigl\{\{\ell,j\} \mid \ell\neq j, \{\ell,i\},\{i,j\}\in
E(G_{i-1})\bigr\},\]
\ie we eliminate vertex $i$ and make a clique out of the neighbours of
$i$ in $G_{i-1}$. The width of the elimination order is the
maximum size of a clique over all~$G_i$ minus one. The elimination width of $G$
is the minimum width over all possible widths of elimination orders of
$G$. It is well known that the treewidth of $G$ is equal to its
elimination width. Let $L'$ be the reverse to~$L$. An easy induction
shows that the neighbours of a vertex $i$ in $G_{i-1}$ are exactly
those of $\Sreach_n[G,L',i]\setminus\{i\}$. It follows that $\col_n(G) = \tw(G)+1$.

Furthermore, it was
shown that the generalised colouring numbers are strongly related, \ie $\col_r(G)\leq (\adm_r(G)-1)\cdot (\adm_r(G)-2)^{r-1}+1$
and $\wcol_r(G)\leq \adm_r(G)^r$ (see for example~\cite{dvovrak13}, but note that in that work, paths of length $0$ are not considered for the $r$-admissibility).

\medskip

\section{Admissibility and Density of Shallow Minors}
\label{sec:colourings}

A graph $H$ with vertex set $V(H)=\{v_1,\ldots, v_n\}$ is a
\emph{minor} of a graph $G$, written $H\minor G$, if there are
pairwise vertex disjoint connected subgraphs $H_1,\ldots, H_n$ of~$G$
such that whenever $v_iv_j\in E(H)$, then there are $u_i\in V(H_i)$
and $u_j\in (H_j)$ with $u_iu_j\in E(G)$. We call
$(H_1,\ldots, H_n)$ a {\em{minor model}} of~$H$ in~$G$. For 
$r\in\N$, the graph $H$
is a {\em{depth-$r$ minor}} of $G$, denoted $H\minor_rG$, if there is
a minor model $(H_1,\ldots,H_n)$ of~$H$ in $G$ such that each $H_i$
has radius at most~$r$.

For $r\in\N$, an \emph{$r$-subdivision} of a graph $H$ is obtained
from $H$ by replacing edges by pairwise internally disjoint paths of
length at most $r+1$. If a graph $G$ contains a $2r$-subdivision of~$H$ 
as a subgraph, then $H$ is a \emph{topological depth-$r$ minor}
of~$G$, written $H\minor_r^t G$. Recall that $H$
is a \emph{topological minor} of $G$ (we write $H\minor^t G$) if some
subdivision of $H$ is a subgraph of $G$, that is, if $H\minor_r^t G$
for some $r\in\mathbb N$.

The \emph{edge density} of a graph $G$ is
$\epsilon(G)=|E(G)|/|V(G)|$. Note that the average degree of~$G$ is
$2\epsilon(G)$. A graph is \emph{$k$-degenerate} if every subgraph has
a vertex of degree at most~$k$.  The maximum of the edge densities of
all $H\minor_{r}G$ is known as the \emph{greatest reduced average density}~$\nabla_r(G)$  of $G$ with rank~$r$. Similarly, the maximum of the edge densities of 
all $H\minor_{r}^tG$ is known as the \emph{topological 
greatest reduced average density} 
$\widetilde\nabla_r(G)$  of $G$ with rank~$r$. We will also refer to the functions 
$r \mapsto \nabla_r(G)$ and $r \mapsto \widetilde\nabla_r(G)$ as \emph{expansion} and \emph{topological expansion} of~$G$, respectively.
As proved in~\cite{dvovrak2007asymptotical}, these measures 
satisfy $\widetilde\nabla_r(G)\leq \nabla_r(G)\leq 4\big(4\widetilde\nabla_r(G)\big)^{(r+1)^2}$. 
  
\medskip
A class $\CCC$ of graphs is \emph{nowhere dense} if for all
$\epsilon>0$ and all $r\in\N$ there is an $n_0\in \N$ such that all
$n$-vertex graphs $G\in\CCC$ with at least $n_0$ vertices satisfy
$\nabla_r(G)\leq n^\epsilon$. $\CCC$ is said to have
\emph{bounded expansion} if for every $r$ there is a $c(r)$ such that
$\nabla_r(G)\le c(r)$ for all $G\in\CCC$. It is easy to see
that all classes of bounded expansion are nowhere dense; the converse
does not hold. We say that $\CCC$ has \emph{polynomial expansion}
if there is a polynomial $p(x)$ such that $\nabla_r(G)\le p(r)$ for all
$r\in\N$ and $G\in\CCC$. 


\medskip
The following theorem implies improvements of previous results from
Kierstead and Yang~\cite{kierstead03} and Zhu \cite{zhu09} to the
exponent of their upper bounds for colouring numbers and the weak
colouring numbers.

\begin{theorem}\label{thm:adm}
  Let $G$ be a graph and let $r\geq 1$. Then
$\adm_r(G)\leq 6r\bigl(\widetilde\nabla_{r-1}(G)\bigr)^3$.
\end{theorem}

Every class $\CCC$ that excludes a topological minor has $\widetilde\nabla_{r}(G)$ 
bounded by a universal constant for every $G\in\CCC$. This includes familiar classes such as classes of bounded
degree, bounded genus, and bounded treewidth. We obtain the following
corollary.

\begin{corollary}
  Let $\CCC$ be a graph class that excludes some fixed graph as a
  topological minor. Then for all $G\in\CCC$ we have $\adm_r(G)\in\Oof(r)$ and 
  $\wcol_r(G)\in\Oof((c_{\CCC}\cdot r)^r)$, where $c_{\CCC}$ is a constant depending only on the class $\CCC$.
\end{corollary}

For the proof of Theorem~\ref{thm:adm} we need a lemma which is a
variation of a result of Dvo{\v{r}}{\'a}k~\cite{dvovrak13}. For a set
$S\subseteq V(G)$ and $v\in S$, let $b_r(S,v)$ be the maximum number
$k$ of paths $P_1,\ldots, P_k$ of length at most $r$ from~$v$ to $S$
with internal vertices in $V(G)\setminus S$ and with
$V(P_i)\cap V(P_j)=\{v\}$ for $1\leq i\neq j\leq k$.

\begin{lemma}[\cite{dvovrak13}]\label{lem:adm}
  For all graphs $G$ and $r\in\mathbb N$, there exists a set
  $S\subseteq V(G)$ such that $b_r(S,v) \geq \adm_r(G)$ for all $v\in S$.
\end{lemma}

\begin{proof}
  Assume that all subsets $S\subseteq V(G)$ contain a vertex $v$ such that
  $b_r(S,v)<\adm_r(G)$. We construct an order
  $L(v_1)< L(v_2)< \ldots < L(v_n)$ of $V(G)$ as follows. If
  $v_{i+1},\ldots, v_n$ have already been ordered, choose $v_{i}$ such
  that if $S_i=\{v_{1},\ldots, v_i\}$, then $b_r(S_i,v_{i})$ is
  minimal. Clearly, the $r$-admissibility of the resulting order is
  one of the values $b_r(S_i,v_i)$ occurring in its construction. This
  implies $\adm_r(G) < \adm_r(G)$, a contradiction.
\end{proof}

\begin{proof}[Proof of Theorem~\ref{thm:adm}]
  Let $G$ be a graph with $\widetilde\nabla_{r-1}(G)\le c$, and let
$\ell\coloneqq6rc^3+1$. Suppose for contradiction that
  $\adm_r(G)>\ell$. By Lemma~\ref{lem:adm}, there exists a set $S$
  such that $b_r(S,v)>\ell$ for all $v\in S$. For $v\in S$, let
  $\mathcal{P}_v$ be a set of paths from $v$ to $S$ witnessing this, and
  let $s\coloneqq|S|$.

  Choose a maximal set $\mathcal{P}$ of pairwise internally
  vertex-disjoint paths of length at most $2r-1$ connecting pairs of
  vertices from $S$ whose internal vertices belong to
  $V(G)\setminus S$ such that each pair of vertices is connected by at
  most one path. Let $H$ be the graph with vertex set $S$ and edges
  between all vertices $v,w\in S$ connected by a path in $\mathcal
  P$.
  Then $H\minor_{r-1}^tG$ and hence
  $|\mathcal{P}|=|E(H)|\leq s\cdot c$.  Let~$M$ be the set of all
  internal vertices of the paths in $\mathcal{P}$, and let
  $m\coloneqq|M|$. Then $m\leq s\cdot c\cdot (2r-2)$.

  Note that we not only have $H\minor_{r-1}^tG$, but also
  $H'\minor_{r-1}^tG$ for all $H'\subseteq H$. Thus for all 
  $H'\subseteq H$ we have
  $\epsilon(H')\le c$, and therefore $H'$ has a vertex of degree at
  most $2c$. In other words, $H$ is $2c$-degenerate. This implies
  that $H$ is $(2c+1)$-colourable and hence contains an independent
  set $R$ of size at least $\left\lceil s/(2c+1)\right\rceil$.
  

  For $v\in S$, let $\mathcal Q_v$ be the set of initial
  segments of paths in $\mathcal{P}_v$ from $v$ to a vertex in
  $(M\cup S)\setminus\{v\}$ with all internal vertices in
  $V(G)\setminus(M\cup S)$. Observe that for $u,v\in R$ the paths in
  $\mathcal Q_v$ and $\mathcal Q_u$ are internally disjoint, because
  if $Q\in\mathcal Q_u$ and $Q'\in\mathcal Q_v$ had an internal vertex
  in common, then $Q\cup Q'$ would contain a path of length at most
  $2r-2$ that is internally disjoint from all paths in $\mathcal P$,
  contradicting maximality of $\mathcal P$.

  Let $G'$ be the union of all paths in $\mathcal{P}$ and all paths in
  $\mathcal Q_v$ for $v\in R$, and let~$H'$ be obtained from $G'$ by
  contracting all paths in $\bigcup_{v\in R}\mathcal Q_v$ to single
  edges. Then $H'\preceq^t_{r-1}G$.

  We have
  $|V(H')|\leq s+m\leq s+s\cdot c\cdot(2r-2)\leq s\cdot c\cdot (2r-1)$
  and at least
  $|E(H')|\geq\left\lceil s/(2c+1)\right\rceil\cdot \ell$
  edges.  Thus $\epsilon(H')\ge \ell/6rc^2>c$. A contradiction.
\end{proof}

\section{The Weak Colouring Numbers of Graphs of Bounded Treewidth}\label{sec:boundedtreewidth}
 
A \emph{tree decomposition} of a graph $G$ is a pair $(T,X)$, where
$T$ is a tree, $X=(X_t:t\in V(T))$, is a family of subsets of V(G)
(called bags) such that
\smallskip
\begin{enumerate}
\item[(i)] $\bigcup_{t\in V(T)}X_t=V(G)$,
\item[(ii)] for every edge 
$\{u,v\}$ of $G$ there exists $t \in V(T)$ with $u,v \in X_t$ and
\item[(iii)] if
$r,s,t \in V(T)$ and $s$ lies on the unique path of $T$ between $r$ and $t$,
then $X_r\cap X_t\subseteq X_s$.
\end{enumerate}  
 
 \smallskip
A graph has \emph{treewidth} at most $k$ if it admits a
tree decomposition $(T,X)$ such that $|X_t|\leq k+1$ for each $t\in
V(T)$ and we write $\tw(G)$ for the treewidth of $G$.  We assume
familiarity with the basic theory of tree decompositions as in~\cite{Diestel}.

It is well known that a graph of treewidth~$k$ has a
tree decomposition $(T,X)$ of width~$k$ such that for every
$\{s,t\}\in E(T)$ we have $|X_s\setminus X_t| \le 1$. We call such
decompositions \emph{smooth}. The following separation property of
tree decompositions is well known.

\begin{lemma}\label{lem:sep}
  If $r,s,t\in V(T)$, $u\in X_r$ and $v\in X_t$ and $s$ is on the
 path of $T$ between $r$ and~$t$, then every path from $u$ to
  $v$ in $G$ uses a vertex contained in~$X_s$.
\end{lemma}

For a tree decomposition $(T,X)$ of $G$ and a node $s\in
V(T)$ we define a partial order $L^{T,s}$ on $V(T)$ demanding that $L^{T,s}(t)\le
L^{T,s}(u)$ if $t$ lies on the path from $s$ to~$u$ (\ie $L^{T,s}$ is
the standard tree order where $s$ is minimum). 

\begin{theorem}\label{uwcoltw}
Let $\tw(G) \leq k$. Then $\wcol_r(G) \leq {r+k \choose k}$.
\end{theorem}

\begin{proof}
  Let $(T,X)$ be a smooth tree decomposition 
  of~$G$ of width at most~$k$.  Since if $G'$ is a
  subgraph of $G$, then $\wcol_r(G')\le\wcol_r(G)$, \WLOG we may assume that
  $G$ is edge maximal of treewidth $k$, \ie each bag induces a
  clique in~$G$. We choose an arbitrary root $s$ of $T$ and let~$L'$
  be some linear extension of $L^{T,s}$.  For every $v\in V(G)$, let $t_v$
  be the unique node of $T$ such that $L'(t_v)=\min\{L'(t)| v\in
  X_t\}$ and define a linear ordering $L \coloneqq L_G^{T,s}$ of $V(G)$ such that: 
  
  \smallskip
  \begin{itemize}
  \item[(i)] $L'(t_v)<L'(t_u) \Rightarrow L(v)<L(u)$, and 
  \item[(ii)] if   $L'(t_v)=L'(t_u)$ (which is possible in the root bag $X_s$), 
  break ties arbitrarily.
  \end{itemize}

\smallskip  
  Fix some $v\in V(G)$ and let $w\in \Wreach_r[G,L,v]$. By
  Lemma~\ref{lem:sep} and the definition of $L$, it is immediate that $t_w$
  lies on the path from $t_v$ to $s$ in $T$.  Let $u \in X_{t_v}$ be
  such that $L(u)\leq L(u')$ for all $u'\in X_{t_v}$.  If $t_v = s$,
  then $|\Wreach_r[G,L,v]| \le k+1$ and we are done.  Otherwise, as
  the decomposition is smooth, $L'(t_u) < L'(t_v)$. We define two
  subgraphs~$G_1$ and~$G_2$ of~$G$ as follows. The graph~$G_1$ is
  induced by the vertices from the bags between~$s$ and $t_u$, \ie by
  the set $\bigcup\{X_t \in V(T) \mid L^{T,s}(t) \le
  L^{T,s}(t_u)\}$. The graph $G_2$ is induced by $\bigcup\{X_t \in\
  V(T) \mid L^{T,s}(t_u) \le L^{T,s}(t) < L^{T,s}(t_v)\}\setminus
  V(G_1)$.

\pagebreak  Let $L_i$ be the restriction of $L$ to $V(G_i)$, for $i=1,2$,
  respectively. We claim that if $w\in \Wreach_r[G,L,v]$, then $w\in
  \Wreach_{r-1}[G_1,L_1,u]\cup \Wreach_r[G_2,L_2,v]$.  To see this,
  let $P=(v=v_1,\ldots, v_\ell=w)$ be a shortest path between $v$ and
  $w$ of length $\ell\le r$ such that $L(w)$ is minimum among all
  vertices of $V(P)$. 
  
  We claim that $L(v_1)>\ldots >L(v_\ell)$ (and call $P$ a decreasing path). This implies in particular that all $t_{v_i}$ lie on the path from $t_v$ to $s$ and that $L^{T,s}(t_{v_1})\geq \ldots \geq L^{T,s}(t_{v_\ell})$ (non-equality may only hold in the last step, if we take a step in the root bag). 

  Assume that the claim does not hold and let $i$ be the first position with $L(v_i)<L(v_{i+1})$. It suffices to show that we can find a subsequence (which is also a path in~$G$) $Q=v_i,v_j,\ldots, v$ of $P$ with $j>i+1$. 
By definition of $t_{v_{i+1}}=:t$, $X_t$ contains~$v_i$. (Indeed, there is an edge between $v_i$ and $v_{i+1}$, which must be contained in some bag, but $v_{i+1}$ appears first in $X_t$ counting from the root and each bag induces a clique in $G$). Let $t'$ be the parent node of~$t$. $X_{t'}$ also contains~$v_i$, as the decomposition is smooth and $v_{i+1}$ is the unique vertex that joins~$X_t$. 
But by Lemma~\ref{lem:sep}, $X_{t'}$ is a separator that separates $v_{i+1}$ from all vertices smaller than $v_{i+1}$. We hence must visit another vertex $v_j$ from $X_{t'}$ in order to finally reach $v$. We can therefore shorten the path as claimed.

If $L(w) \le L(u)$, then $P$ goes through $X_{t_u}$ by Lemma~\ref{lem:sep}. Let $u'$ be the first vertex of~$P$ that lies in
$X_{t_u}$. We show that there is a shortest path from $v$ to~$u'$ that
uses~$u$ as the second vertex. By assumption, $v\neq u$. If $\{v,u'\}\in E(G)$, then
$\{v,u'\}$ must be contained in some bag~$X_{t'}$. By definition of $t_v$,
$t' = t_v$, as $t_v$ is the first node of $T$ on the path from $s$ to
$t_v$ containing~$v$. By definition of $t_u$ and because $(T,X)$ is
smooth, $u$ is the only vertex from $t_v$ that appears in $t_u$. Thus
$u' = u$, so the shortest path from $v$ to $u'$ uses $u$. If the
distance between $v$ and~$u'$ is at least~$2$, a shortest path can be
chosen as $v,u,u'$. Indeed $u\in X_{t_u} \cap X_{t_v}$ and every bag induces a
clique by assumption.

It follows that if $L(w) \le L(u)$ and $w\in \Wreach_r[G,L,v]$, then there
is a shortest path from $v$ to $w$ that uses $u$ as the
second vertex. Thus $w\in\Wreach_{r-1}[G_1,L_1,u]$, as $P$ is decreasing.

If $L(w) > L(u)$, then $P$ never visits vertices of $G_1$. If $P$
lies completely in~$G_2$, we have $w\in \Wreach_r[G_2,L_2,v]$. If $P$
leaves $G_2$, it visits vertices of $G$ that are contained only in
bags strictly below $t_v$. However, this is impossible, as $P$ is decreasing.

Therefore we have 
$$|\Wreach_r[G,L,v]|\leq |\Wreach_{r-1}[G_1,L_1,u]|+|\Wreach_r[G_2,L_2,v]|. \ \ \ \ \ (\star)$$
Note that the tree\-width of $G_2$ is at most $k-1$, as we removed $u$ from every bag. More
precisely, the tree decomposition $(T^2,X^2)$ of $G_2$ of width at
most $k-1$ is the restriction of $(T,X)$ to $G_2$, \ie we take tree
nodes $t$ contained between $t_u$ and $t_v$ (including $t_v$ and not
including $t_u$) and define $X^2_t = X_t\cap V(G_2)$.

Now, recall the definition of $L_G^{T,s}$ as in the beginning of the proof and let 
$w(r,k)$ be the maximum $|\Wreach_r[H,L_H^{T,s},v]|$, ranging over all graphs $H$ with $\tw(H)\leq k$,
linear orders $L_H^{T,s}$ obtained by an $s\in V(T)$, and vertices $v\in V(H)$. By $(\star)$, we then have
$|\Wreach_r[G,L,v]|\leq w(r,k-1)+w(r-1,k)$. Since $G,L$ and $v$ where arbitrary, it follows that
$$w(k,r) \leq w(k,r-1)+w(k-1,r).$$

Recall that $\wcol_1(G)$ equals the degeneracy
of $G$ plus one and note that every graph of tree\-width  $\leq k$ is
$k$-degenerate, hence $w(k,1)\leq k+1$. Furthermore, it is easy to observe that 
for a tree $T$, we have $\wcol_r(T)\leq r+1$: any linear extension of a tree-order with respect to some root will do. 
Hence $w(1,r)\leq r+1$. Since $\binom{r+k}{k}=\binom{r+k-1}{k}+\binom{r+k-1}{k-1}$, 
we conclude by induction that $w(r,k) \leq \binom{r+k}{k}$.
\end{proof}

The proof of Theorem~\ref{uwcoltw} gives rise to a construction of a
class of graphs that matches the upper bound proven there. We
construct a graph of treewidth $k$ and weak $r$-colouring number
${k+r\choose k}$ whose tree decomposition $(T,X)$ has a highly
branching host tree $T$. This enforces a path in the tree from the root to
a leaf that realises the recursion from the proof of Theorem~\ref{uwcoltw}.

\begin{theorem}\label{thm:Gkr}
  For every $k\geq 1$, $r\geq 1$, there is a family of graphs $G^k_r$ with $\tw(G^k_r) = k$, such
  that $\wcol_{r}(G^k_r)={r+k \choose k}$. In fact, for all $r'\le
  r$, $\wcol_{r'}(G^k_r) = {r'+k \choose k}$.
\end{theorem}
\begin{proof}
Fix $r,k$ and let $c={r+k\choose k}$. We define graphs $G(k',r')$ for all $r'\leq r, k'\leq k$ and corresponding tree decompositions $\mathcal{T}(k',r') = (T(k',r'),X(k',r'))$ of $G(k',r')$ of width~$k'$ with a
  distinguished root $s(T(k',r'))$ by induction on~$k'$ and~$r'$. We will show that $\wcol_{r'}(G(k',r'))\geq {r'+k'\choose k'}$.  We guarantee several invariants for all values of $k'$ and~$r'$ which will give us control over a sufficiently large part of any order that witnesses $\wcol_{r'}(G(k',r'))\geq {r'+k'\choose k'}$.
  \smallskip
  
  \begin{enumerate}
  \item[(i)] There is a bijection $f:V(T(k',r'))\rightarrow V(G(k',r'))$ such that $f(s(T(k',r')))$ is the unique vertex contained in $X_{s(T(k',r'))}$ and if $t$ is a child of $t'$ in $T(k',r')$, then $f(t)$ is the unique vertex of $X_{t}\setminus X_{t'}$. Hence any order defined on $V(T)$ directly translates to an order of $V(G)$ and vice versa. 
  \item[(ii)] In any order $L$ of $V(G(k',r'))$ which satisfies $\wcol_r(G(k',r'))\leq c$, there is some root-leaf path $P=t_1,\ldots, t_m$ such that $L(f(t_1))<\ldots<L(f(t_m))$. 
  \item[(iii)] Every bag of $T(k',r')$ contains at most $k'+1$ vertices. 
\end{enumerate}

\smallskip
It will be convenient to define the tree decompositions first and to
define the corresponding graphs as the unique graphs induced by the
decomposition in the following sense. For a tree $T$ and a family of
finite and non-empty sets $(X_t)_{t\in V(T)}$ such that if $z,s,t\in
V(T)$ and~$s$ is on the path of $T$ between $z$ and~$t$, then $X_z\cap X_t\subseteq X_s$, we define the graph \emph{induced} by $(T,(X_t)_{t\in V(T)})$ as the graph~$G$ with $V(G)=\bigcup_{t\in V(T)}X_t$ and $\{u,v\}\in E(G)$ if and only if $u,v\in X_t$ for some $t\in V(T)$. Then $(T,(X_t)_{t\in V(T)})$ is a tree decomposition of $G$. 

\smallskip

For $k'\geq 1, r'=1$, let $T(k',r')\eqqcolon T$ be a tree of depth
$k'+1$ and branching degree~$c$ with root $s$. Let $L^{T,s}$ be the
natural partial tree order.  Let $f\colon V(T) \to V$ be a bijection
to some new set $V$.
We define $X_{t}\coloneqq
\{f(t) \mid L^{T,s}(t')\leq L^{T,s}(t)\}$. Let $G(k',r')$ be the
graph induced by the decomposition. The first and the
third invariants clearly hold. For the second invariant, consider a
simple pigeon-hole argument. For every non-leaf node $t$, the vertex
$f(t)$ has $c$ neighbours $f(t')$ in the child bags $X_{t'}$ of
$t$. Hence some $f({t'})$ must be larger in the order. This
guarantees the existence of a path as required. 

For $k'=1, r'\geq 1$, let $T(k',r')\eqqcolon T$ be a tree of depth $r'+1$ and
branching degree~$c$ with root $s$ and let $f$ be as before. Let $X_s:=\{f(s)\}$ and for each
$t'\in V(T)$ with parent $t\in V(T)$ let $X_{t'}:=\{f(t),f(t')\}$. Let
$G(k',r')$ be the graph induced by the decomposition. All invariants
hold by the same arguments as above. Note that~$G^1_1$ is the same
graph in both constructions and is hence well defined.

Now assume that $G(k',r'-1)$ and $G(k'-1,r')$ and their respective tree decompositions have been defined. 
Let $T(k',r')$ be the tree which is obtained by attaching~$c$ copies of $T(k'-1,r')$ as children to each leaf of $T(k',r'-1)$. We define the bags that belong to the copy of $T(k',r'-1)$, exactly as those of $T(k',r'-1)$. 
To every bag of a copy of $T(k'-1,r')$ which is attached to a leaf~$z$, we add $f'(z)$ (where $f'$ is the bijection from $T(k',r'-1)$). Let $G(k',r')$ be the graph induced by the decomposition. 

It is easy to see how to obtain the new bijection $f$ on the whole graph such that it satisfies the invariant. It is also easy to see that each bag contains at most $k'+1$ vertices.
For invariant~(ii), let $P_1=t_1,\ldots, t_m$ be some root-leaf path in $T(k',r'-1)$ which is ordered such that $L(f(t_1))<\ldots< L(f(t_m))$. Let $v=f(t_m)$ be the unique vertex in the leaf bag in which~$P_1$ ends. By the same argument as above, $v$ has many neighbours $s'$ such that $f^{-1}(s')$ is a root of a copy of $T(k'-1,r')$. One of them must be larger than $v$. In an appropriate copy we find a path $P_2$ with the above property by assumption. We attach the paths to find the path $P=t_1\ldots t_\ell$ in~$T(k',r')$. 

We finally show that $\Wreach_{r'}[G(k',r'),L,f(t_\ell)]={r'+k'\choose k'}$. This is again shown by an easy induction. Using the notation of the proof of Theorem~\ref{uwcoltw}, we observe that the graph $G_1$ is isomorphic to $G(k',r'-1)$ in $G(k',r')$ and $G_2$ is isomorphic to $G(k'-1,r')$. Furthermore we observe that the number of vertices reached in these graphs are exactly $w(k',r'-1)$ and $w(k'-1,r')$, so that the upper bound is matched. Similarly one shows that $\wcol_{r'}(G(k,r))={r'+k\choose k}$. The theorem follows by letting $G^k_r\coloneqq G(k,r)$.
\end{proof}

It is proven in \cite{kierstead03,NesetrilOdM12} that for every graph
$G$, $\wcol_r(G)\leq (\col_r(G))^r$. To our knowledge, there is no
example in the literature that verifies the exponential gap between
$\wcol_r$ and $\col_r$. As $\col_r(G)\le \tw(G)$ and $G_r^k$ contains
a $k+1$-clique, Theorem~\ref{thm:Gkr} provides an example that is close to an
affirmative answer for arbitrarily large generalised colouring numbers, in a rather uniform manner.

\begin{corollary}\label{colgap}
   For every $k\geq 1$, $r\geq 1$, there is a graph $G^k_r$ such that for
  all $1\le r'\le r$ we have
  $\col_{r'}(G^k_r)=k+1$ and $\wcol_{r'}(G^k_r) \geq \big(\frac{\col_{r'}(G^k_r)}{r'}\big)^{r'}$.
\end{corollary}

 \begin{proof}
  Since $\col_{r'}(G^k_r)=k+1$, we have
 $$\wcol_{r'}(G^k_r)={r'+k \choose k}={k+r' \choose r'}\geq
  \left(\frac{k+r'}{r'}\right)^{r'}\geq
  \left(\frac{\col_{r'}(G^k_r)}{r'}\right)^{r'}.$$
 \end{proof}

As a further application of Theorem~\ref{thm:Gkr}, we now construct a class of graphs with polynomial expansion that has super-polynomial
weak colouring numbers. For a graph~$G$ denote by $G^{(r)}$ the exact $r$-subdivision
of $G$, that is, the graph obtained from~$G$ by replacing every edge by a path of length
$r+1$ (with $r$ vertices on it). 

\begin{theorem}\label{PolyExpWeak}
The class $\CCC=\{G^{(6\tw(G))}~:~G$ graph$\}$ has polynomial expansion and
super-polynomial weak colouring numbers.  
\end{theorem}
\begin{proof}
Let $r\in \N$ and let $G$ be a graph of treewidth $t\coloneqq \tw(G)$. Let $H$ be
the densest depth-$r$ minor of $G^{(6t)}$. If $r\geq t$, we conclude that 
$\epsilon(H)\leq t\leq r$, since~$G$ as a graph of treewidth $t$ is $t$-degenerate
and so are all its minors. On the other hand, if $r<t$, it is easy to see that every vertex 
of $H$ of degree greater than $2$ is adjacent only to vertices of degree at most $2$. 
Hence in this case $H$ is $2$-degenerate. We conclude that $\nabla_r(G^{(6t)})\leq 
r+2$, and hence~$\CCC$ is a class of polynomial (and even linear) expansion. 

To prove that $\CCC$ has super-polynomial weak colouring numbers, we first relate
the weak colouring numbers of the exact $k$-subdivision for any $k\in \N$ 
of a graph $G$ to the weak 
colouring numbers of $G$. We claim that 
$\wcol_{r\cdot (k+1)}(G^{(k)})\geq 
\frac{1}{2}\wcol_r(G)$. 

\pagebreak
To see this, consider an arbitrary order $L$ of $V(G^{(k)})$. We construct the 
following order~$L'$ of $V(G)$. For $u\in V(G)$ let $u_m\in V(G^{(k)})$ be 
the $L$-minimal vertex of $N_{k-1}(u)$ in~$G^{(k)}$. For $u,v\in V(G)$
we let $L'(u)<L'(v)$ if $L(u_m)<L(v_m)$ or if $L(u_m)=L(v_m)$ and $L(u)<L(v)$. 
Now fix some $v\in V(G)$ and consider the set $\Wreach_r[G,L',v]$. We observe
that for each $u\in \Wreach_r[G,L',v]$ the vertex $u_m$ is weakly $(r\cdot (k+1))$-reachable
from $v$ in $G^{(k)}$ with respect to $L$: we traverse the subdivided path $P$
from $v$ to $u$ in $G^{(k)}$ and, if we have not visited $u_m$ yet, append the path from $u$ to $u_m$. This path
has length at most $k\cdot r+(k-1)\leq r\cdot (k+1)$ and by definition of $L'$ and
the fact that $u$ is $L'$ minimal on the corresponding path in~$G$, 
$u_m$ is $L$-minimal on $P$. Now when counting weakly reachable vertices, 
$u_m$ may be counted twice, once for each end of the subdivided edge it lies on. 
We conclude that $2\cdot |\Wreach_{r\cdot(k+1)}[G^{(k)},L,v]|\geq |\Wreach_r[G,L',v]|$
for all $v\in V(G)$. As~$L$ was chosen arbitrarily, it follows that 
$\wcol_{r\cdot (k+1)}(G^{(k)})\geq 
\frac{1}{2}\wcol_r(G)$.

As we have proved in Theorem~\ref{thm:Gkr} for every $k,r\in \N$ 
there exists a graph $G_r^k$ of treewidth~$k$ that satisfies $\wcol_r(G_r^k)={r+k \choose r}$. Now for each $t\in\N$
we apply the theorem to $k=r=t$, yielding a graph $H\coloneqq G_t^t$ with 
$\wcol_t(H)=\binom{2t}{t}$. 
With our above observation we conclude that 
$\wcol_{t\cdot (6t+1)}(H^{(6t)})\geq \frac{1}{2}\wcol_t(H)=\frac{1}{2}{2t\choose t}\geq \frac{1}{2}\cdot\frac{4^t}{\sqrt{\pi t}}(1-\frac{1}{8t})$. 
Substituting $\frac{\sqrt{(1+24r)}-1}{12}$ for~$t$ gives us $\wcol_r(\CCC)\in 
\Omega((4-\epsilon)^{\frac{\sqrt{6r}}6})$ for any $\epsilon>0$, which is super-polynomial in $r$. 
\end{proof}

Theorem~\ref{PolyExpWeak} answers negatively the question whether graph classes of polynomial
expansion have polynomial weak colouring numbers. This question is 
attributed to Joret and Wood in~\cite{EsperetR18}. They also ask whether
graph classes of polynomial
expansion have polynomial strong colouring numbers. This was proved in the
meantime for
graphs excluding a fixed minor~\cite{HeuvelMQRS17}. 

\section{High-Girth Regular Graphs}\label{sec:reghighgirth}

The goal of this section is to study the generalised colouring numbers of graph classes with constant topological expansion (such as classes excluding a topological minor).
In light of~(\cite{zhu09}, Lemma 3.3), it is not surprising that such classes can have exponential weak colouring numbers. 
Surprisingly, we prove that, in fact, even classes of bounded degree (which are of the simplest classes that can exclude a topological minor) have superpolynomial colouring numbers, too. 
For this section, we let $n\coloneqq |V(G)|$.

\begin{theorem}\label{coldreg}
  Let $G$ be a $d$-regular graph of girth at least $4g+1$, where
  $d\geq 7$. Then for every $r\leq g$, 
  $$\col_r(G)\geq  \frac{d}{2}\left(\frac{d-2}{4}\right)^{2^{\lfloor logr \rfloor}-1}.$$
\end{theorem}

\begin{proof}
For an ordering $L$ of $G$, let $R_r(v)=\Sreach_r[G,L,v]\setminus
  \Sreach_{r-1}[G,L,v]$ and 
$U_r= \sum_{v\in V(G)}|R_r(v)|$.

Suppose that $r\leq g$ and notice that for $u,w \in R_r(v)$, we have
that either $u \in R_{2r}(w)$ or $w \in R_{2r}(u)$.  Therefore, every
vertex $v\in V(G)$ contributes at least ${|R_r(v)| \choose 2}$ times
to $U_{2r}$.  Moreover, since $r\leq g$, for every $u,w$ with $u \in
R_{2r}(w)$ there is at most one vertex $v\in V(G)$ such that $u,w \in
R_r(v)$ (namely the middle vertex of the unique $(u,v)$-path of length
$2r$ in $G$). It follows that for every $r\leq g$,
 \begin{align*}
   U_{2r} &\geq \sum\limits_{v\in V(G)}{|R_r(v)| \choose 2}
   = \frac{1}{2}\sum\limits_{v\in V(G)}|R_r(v)|^2-\frac{1}{2}\sum\limits_{v\in V(G)}|R_r(v)|\\
   &\geq \frac{1}{2n}\Big(\sum\limits_{v\in V(G)}|R_r(v)|\Big)^2-\frac{1}{2}U_r 
   = \frac{1}{2n}U_r^2-\frac{1}{2}U_r
 \end{align*}
where for the second inequality we have used the Cauchy-Schwarz
inequality.

Let $c_r=\frac{U_r}{n}.$ Then for every $r\leq g,$ we obtain
$c_{2r}\geq \frac{1}{2}c_r\left(c_r-1\right)$.  But, 
$$U_1=\sum_{v\in V(G)}|\Sreach_1[G,L,v]\setminus\{v\}|=\frac{1}{2}dn\,,$$ so that
  $c_1=\frac{d}{2}>3$, since $d\geq 7$.  By induction and because
  $c_{2r}\geq \frac{1}{2}c_r\left(c_r-1\right)$, for every $r=2^{r'}\leq g$
    we have $c_{2r}\geq c_r\geq 3\,$. Therefore $c_r\geq c_1=\frac{d}{2}$. 
Again because $c_{2r}\geq \frac{1}{2}c_r\left(c_r-1\right)$, for every $r=2^{r'}\leq g$ we have
      $$c_{2r}\geq \frac{1}{2}c_r^2-\frac{1}{2}c_r\geq \frac{1}{2}c_r^2-\frac{1}{d}c_r^2=\frac{d-2}{2d}c_r^2\,.$$ 
Then  for every $r=2^{r'}\leq g$, it easily follows by induction that $c_r\geq \frac{d}{2}\Big(\frac{d-2}{4}\Big)^{r-1}\,.$

Finally, let $C_r=\frac{1}{n}\sum_{v\in V(G)}|\Sreach_r[G,L,v]|$. Then, $C_r=\sum_{i=1}^rc_i$. In particular, it is
$C_r\geq c_{2^{\lfloor logr \rfloor}} \geq \frac{d}{2}\left(\frac{d-2}{4}\right)^{2^{\lfloor logr \rfloor}-1}$,
and hence for every $r\leq g$ there exists a vertex $v_r \in V(G)$
such that
$|\Sreach_r[G,L,v_r]| \geq
\frac{d}{2}\left(\frac{d-2}{4}\right)^{2^{\lfloor logr \rfloor}-1}$.
Since $L$ was arbitrary, the theorem follows.
\end{proof}

Unfortunately, our proof above makes sense only if $d\geq 7$, which is
also best possible with this approach, since for $d\leq6$, we have
$c_1\leq 3$. Then for the recurrence relation $c_{2r}=
\frac{1}{2}c_r^2-\frac{1}{2}c_r$, we get $c_{2^i}\leq c_{2^{i-1}}$ for
every $i$ and we clearly cannot afford to have $c_{2^i}$
non-increasing. Somewhat better constants can be achieved if in the
estimation of $c_r$ one uses that $c_{2^i}\geq c_{2^{i_0}}$, for
$i\geq i_0>0$, instead of the relation $c_{2^i}\geq c_1$, as in
our proof. Since $d\geq 7$ would be still the best that we would be
able to do, we adopted the simpler approach for easier
readability. 

Actually, by combining a known result for the $\nabla_r$ of high-girth 
regular graphs (\cite{DiestRemp04},\cite{NesetrilOdM12} Exercise 4.2) 
and (\cite{zhu09}, Lemma 3.3), we get exponential lower bounds for the 
weak colouring number of high-girth $d$-regular graphs, already for $d\geq 3$. In particular,
for a $3$-regular graph~$G$ of high enough girth, $\wcol_r(G)\geq 3\cdot 2^{\lfloor r/4 \rfloor-1}$.
The methods above can be extended to get appropriate bounds in terms of their degree 
for regular graphs of higher degree, but by adopting a more straightforward approach,
we get better bounds for high-girth $d$-regular graphs for $d\geq 4$.

\begin{theorem}\label{wcoldreg}
  Let $G$ be a $d$-regular graph of girth at least $2g+1$, where
  $d\geq 4$. Then for every $r\leq g$, 
  $$\wcol_r(G)\geq \frac{d}{d-3}\left(\Big(\frac{d-1}{2}\Big)^r-1\right)\,.$$
\end{theorem}

\begin{proof}
Let $L$ be an ordering of $G$. For $u,v \in V(G)$ with $d(u,v)\leq
  r$, let $P_{uv}$ be the unique $(u,v)$-path of length at most $r$,
  due to the girth of $G$.  Let $$Q_r(v)=\Wreach_r[G,L,v]\setminus  \Wreach_{r-1}[G,L,v],$$ 
  and define $S_r= \sum_{v\in V(G)}|Q_r(v)|$.
For $r\leq g-1$, a vertex $u \in Q_r(v)$ and $w \in N(v)\setminus V(P_{uv})$, it holds that
either $w\in Q_{r+1}(u)$ or $u\in Q_{r+1}(w)$.  Notice that
$|N(v)\setminus V(P_{uv})|=d-1$ and that $P_{vu}$ and $P_{uw}$ are
unique.  Therefore, every pair of vertices $v,u$ with $u \in Q_r(v)$
corresponds to at least $d-1$ pairs of vertices $u,w$ with either $u \in
Q_{r+1}(w)$ or $w \in Q_{r+1}(u)$ and hence contributes at least $d-1$
times to $S_{r+1}$.  Since every path of length $r+1$ contains exactly
two subpaths of length $r$, we have for every $r \leq g-1 $ that
$$2S_{r+1}\geq (d-1)S_r.$$ Let $w_r=\frac{S_r}{n}$. Then, for every $r\leq
g-1$ we have $w_{r+1} \geq \frac{d-1}{2}w_r\,.$

But, $$\sum_{v\in V(G)}|\Wreach_1[G,L,v]\setminus \{v\}|=\frac{1}{2}dn\,,$$ so that
$w_1=\frac{d}{2}$. 

It easily follows by induction that for every $r\leq g$, we have
$w_r\geq \frac{d}{2}\left(\frac{d-1}{2}\right)^{r-1}\,.$ Finally, let
$W_r=\frac{1}{n}\sum_{v\in V(G)}|\Wreach_r[G,L,v]|$. Then,
$$W_r=\sum_{i=1}^rw_i \geq \sum_{i=1}^r\frac{d}{2}\left(\frac{d-1}{2}\right)^{i-1}
=\frac{d}{d-3}\left(\left(\frac{d-1}{2}\right)^r-1\right),$$ and hence we have that for
every $r\leq g$ there exists a vertex $v_r \in V(G)$ such that
$|\Wreach_r[G,L,v_r]| \geq \frac{d}{d-3}\big(\left(\frac{d-1}{2}\right)^r-1\big)\,.$ Since $L$ was
arbitrary, the theorem follows.
\end{proof}

\begin{rmrk}
  Notice that for every $d$-regular graph $G$ and every radius $r$, we
  have $adm_r(G)\leq \Delta(G)+1=d+1$, so by Theorem~\ref{coldreg} for
  every $d\geq 7$ and every $r\leq g$, the $d$-regular graphs of girth
  at least $2g+1$ verify the exponential gap between $adm_r,\Delta(G)$
  and $\col_r,\wcol_r$ of the known relations from Section~\ref{sec:prelim}.
\end{rmrk}

\section{Neighbourhood Covers}\label{sec:covers}

Neighbourhood covers of small radius and small size play a key role in the
design of many data structures for distributed systems. For references about
neighbourhood covers, we refer the reader to \cite{abraham2007strong}.

For~$r\in\N$, an \emph{$r$-neighbourhood cover}~$\XXX$ of a graph~$G$ is a set
of connected subgraphs of~$G$ called \emph{clusters}, such that for every
vertex~$v\in V(G)$ there is some~$X\in\XXX$ with~$N_r(v)\subseteq X$.

The \emph{radius}~$\rad(\XXX)$ of a cover~$\XXX$ is the maximum radius of any of
its clusters. The \emph{degree}~$d^\XXX(v)$ of~$v$ in~$\XXX$ is the number of
clusters that contain~$v$. A class $\CCC$ \emph{admits sparse neighbourhood
covers} if for every $r\in\N$, there exists $c\in \N$ such that for all
$\epsilon>0$, there is $n_0\in\N$ such that for all $G\in \CCC$ of order at
least $n_0$, there exists an $r$-neighbourhood cover of radius at most $c\cdot
r$ and degree at most $|V(G)|^\epsilon$. For any graph $G$, one can construct an
$r$-neighbourhood cover of radius $2r-1$ and degree $2k\cdot |V(G)|^{1/r}$ and
asymptotically these bounds cannot be improved \cite{thorup05}. 

\begin{theorem}[Theorem 16.2.4 of \cite{peleg00}, \cite{thorup05}]\label{thm:nocover} 
For every $r$ and $k\geq 3$, there exist infinitely many graphs $G$ for which every $r$-neighbourhood cover of radius at most $k$ has degree $\Omega(|V(G)|^{1/k})$.  
\end{theorem}

For restricted classes of graphs, better covers exist. The most general results
are that a class excluding a complete graph on $t$ vertices as a minor
admits an $r$-neighbourhood cover of radius $\Oof(t^2\cdot r)$ and degree
$2^{\Oof(t)}t!$ \cite{abraham2007strong}, as well as the following result from
\cite{GKS14}.

\begin{theorem}[\cite{GKS14}]\label{thm:alg-covers} 
Let~$\CCC$ be a nowhere dense class of graphs. There is a function~$f$ such that for all~$r\in\N$ and~$ \epsilon>0$ and all graphs~$G\in\CCC$ with~$|V(G)|\geq f(r,\epsilon)$, there exists an~$r$-neighbourhood cover of radius at most~$2r$ and maximum degree at most~$|V(G)|^\epsilon$. More precisely, if $\wcol_{2r}(G)=d$, then  there exists an $r$-neighbourhood cover of radius at most $2r$ and maximum degree at most $d$. 
\end{theorem}

Hence our new bounds for the weak colouring numbers on restricted graph classes
immediately imply improved $r$-neighbourhood covers on these classes. 

We show that for monotone classes the converse of Theorem~\ref{thm:alg-covers} is also true. We first observe that the lower bounds in Theorem~\ref{thm:nocover} come from a well known somewhere dense class. 

\begin{lemma}\label{lem:nocover}
Let $d\geq1, k\geq 2$ and let $G$ be a graph of girth at least $k+1$ and edge density at least $d$. Then every $1$-neighbourhood cover of radius at most $k$ has degree at least $d$. 
\end{lemma}

\begin{lemma}[\cite{lazebnik1995new}]\label{lem:girthexists}
Let $r\geq 5$. There are infinitely many graphs $G$ of girth at least $4r$ with edge density at least $c_0\cdot |V(G)|^{1/(3(r-1))}$ for some constant $c_0>0$. 
\end{lemma}

\begin{theorem}\label{thm:nc} 
 If $\CCC$ is somewhere dense and monotone, then
$\CCC$ does not admit sparse neighbourhood covers. 
\end{theorem} 
\begin{proof}
Let $\CCC$ be somewhere dense. Then for some integer $s$, all graphs $H$ are topological
depth-$s$ minors of a graph $G\in\CCC$. Assume towards a contradiction that~$\CCC$ admits a sparse neighbourhood cover. Then for every $G\in\CCC$ there is
an $r\cdot s$-neighbourhood cover of radius $c\cdot r\cdot s$ (for some constant
$c$) which for every $\epsilon>0$ has degree at most $|V(G)|^\epsilon$ if $G$ is
sufficiently large. Fix some $r\geq 5$. 

\begin{Claim}\label{lem:subdiv}
If an $s$-subdivision of $H$ admits an $r\cdot s$-neighbourhood cover
of radius $c\cdot r\cdot s$ and degree $d$, then $H$ admits an $r$-neighbourhood cover of radius $c\cdot r\cdot s$ and degree $d$. 
\end{Claim}
\begin{proof}
  Let $G$ be an $s$-subdivision of $H$ and let $\XXX$ be an $r\cdot s$-neighbourhood cover of~$G$. Let~$\YYY$ be the \emph{projected cover} which for every $X\in \XXX$ has a cluster $Y(X)\coloneqq X\cap V(H)$

  Then $\YYY$ is an $r$-neighbourhood cover of radius $c\cdot r\cdot s$ and degree $d$: Clearly, every $Y(X)$ is connected and has radius at most $c\cdot r\cdot s$. Let $v\in V(G)$. There is a cluster $X\in \XXX$ such that $N_{rs}^G(v)\subseteq X$. Then $N_r^H(v)= N_{rs}^G(v) \cap V(H)\subseteq X\cap V(H)=Y(X)$. Finally, the degree of $\YYY$ is at most $d$, as every vertex $v$ of $H$ is exactly in those clusters $Y(X)$ with $v\in X$.  
$\hfill\dashv$\end{proof}

Let $H$ be a large graph of girth greater than $c\cdot r\cdot s$ with edge density $d=c_0\cdot |V(H)|^{1/(crs)}$ for some constant $c_0$. Such $H$ exists by Lemma~\ref{lem:girthexists} and~$H$ does not admit an $r\cdot s$-neighbourhood cover of radius $c\cdot r\cdot s$ and degree $d$ by Lemma~\ref{lem:nocover}. As $\CCC$ is monotone, an $s$-subdivision of $H$ is a graph $G\in\CCC$ with $|V(G)|\leq |V(H)|+s\cdot |E(H)|\leq 2c_0s|V(H)|^{1+1/(crs)}$. 

By assumption, $G$ admits an $r\cdot s$-neighbourhood cover of
radius at most $c\cdot r\cdot s$ and degree at most $|V(G)|^{\epsilon}$ for $\epsilon=1/(2crs)$ if $G$ is large enough. It
follows from Claim \ref{lem:subdiv} that $H$ has a cover of radius $c\cdot
r\cdot s$ and degree at most 
\begin{align*}
|V(G)|^{\epsilon}\leq \big(2c_0s|V(H)|^{1+1/(crs)}\big)^{\epsilon}& =(2c_0s)^{\epsilon}\cdot |V(H)|^{\epsilon+\epsilon/(crs)}\\
& <c_0|V(H)|^{2\epsilon}=c_0|V(H)|^{1/(crs)}\end{align*} 
for sufficiently large $H$. A contradiction.  \end{proof}

A similiar characterisation for classes of bounded expansion was found by \NOdM\ \cite{NesetrilM15}.

\section{The Complexity of Computing $\mathbf{wcol_r(G)}$}
\label{sec:NPcomplete}

Unlike computing the degeneracy of a graph $G$, \ie $\wcol_1(G)+1$,
deciding whether $\wcol_r(G)=k$ turns out to be NP-complete for all
$r\ge 3$. The case $r=2$ remains an open question. Clearly, the problem is in NP, hence it remains to show NP-hardness. The proof is a straightforward modification of a proof of Pothen~\cite{pothen88}, showing that computing a minimum elimination tree height problem is NP-complete. 
It is based on a reduction from the NP-complete problem \textsc{Balanced Complete Bipartite subgraph} (BCBS, problem GT24 of \cite{garey02}): given a bipartite graph $G$ and a positive integer $k$, decide whether there are two disjoint subsets $W_1,W_2\subseteq V(G)$ such that $|W_1|=|W_2|=k$ and such that $u\in W_1, v\in W_2$ implies $\{u,v\}\in E(G)$. For a graph $G$, let $\bar G$ be its complement graph.

\begin{lemma}
Let $G=(V_1\cup V_2, E)$ be a bipartite $n$-vertex graph and let
$k\in\N$. Then $G$ has a balanced complete bipartite subgraph with
partitions $W_1, W_2$ of size~$k$ if and only if
$\wcol_r(\bar{G})=\wcol_3(\bar{G})\le n-k$ for all $r\geq 3$. 
\end{lemma}
\begin{proof}
$\bar{G}$ is the complement of a bipartite graph, \ie $V_1$ and~$V_2$ induce complete subgraphs in $\bar{G}$ and there are possibly further edges between vertices of $V_1$ and~$V_2$. Thus, for any two vertices $u,v$ which are connected in $\bar{G}$ by a path $P$, there is a subpath of $P$ between~$u$ and~$v$ of length at most $3$. Hence $\wcol_r(\bar{G})=\wcol_3(\bar{G})$ for any $r\geq 3$ and it suffices to show that~$G$ has a balanced complete bipartite subgraph with partitions $W_1, W_2$ of size $k$ if and only $\wcol_3(\bar{G})=n-k$.

First assume that there are sets $W_1\subseteq V_1, W_2\subseteq V_2$ with $|W_1|=|W_2|=k$ and such that for all $u\in W_1, v\in W_2$ there is an edge $\{u,v\}\in E(G)$. Let $L$ be some order which satisfies $L(u)<L(v)$ if $u\in V(\bar{G})\setminus (W_1\cup W_2)$ and $v\in W_1\cup W_2$ and $L(v)< L(w)$ if $v\in W_1$ and $w\in W_2$. Then any vertex from $V(\bar{G})\setminus (W_1\cup W_2)$ weakly reaches at most $n-2k$ vertices and any vertex from $W_i$ for $1\leq i\leq 2$ weakly reaches at most $n-k$ vertices. 

Now let $L$ be an order with $\Wreach_3[\bar{G},L, v]\leq
n-k$ for all $v\in V(G)$. Assume without loss of generality that $V(G)
= \{v_1,v_2,\ldots,v_n\}$ with $L(v_i)< L(v_{i+1})$ for all $i<n$.
Denote by~$\bar{G}_i$ the subgraph $\bar{G}[\{v_i,\ldots, v_n\}]$ and let
$V_1^i \coloneqq V(\bar{G}_i)\cap V_1$ and $V_2^i\coloneqq V(\bar{G}_i)\cap
V_2$. Let $\ell\geq 1$ be minimal such that there is no edge between
$V_1^\ell$ and $V_2^\ell$ in $\bar{G}$. It exists because one of $V^n_1$ or $V^n_2$ is empty.
Clearly, $V_1^\ell$ and
$V_2^\ell$ induce a complete bipartite graph in $G$. Let $j_1\coloneqq|V_1^\ell|$
and $j_2\coloneqq|V_2^\ell|$. We show that $j_1,j_2\geq k$. It is easy to see that
$\Wreach_3[\bar{G},L, w_1]\leq \ell+j_1$ for the maximal element $w_1\in
V_1^\ell$ and $\Wreach_3[\bar{G},L, w_2]\leq \ell+j_2$ for the maximal element
$w_2\in V_2^\ell$. We have $j_1+j_2=n-\ell$ and, without loss of generality,
$\ell+j_1\leq \ell+j_2\leq n-k$. Hence $j_1\leq j_2\leq n-\ell-k=j_1+j_2-k$,
which implies both $j_1\geq k$ and $j_2\geq k$.
\end{proof}

The above reduction is polynomial time computable, so we obtain the following theorem. 

\begin{theorem}
Given a graph $G$ and $k,r\in\N, r\geq 3$, it is NP-complete to decide whether $\wcol_r(G)=k$.
\end{theorem}

\section{Concluding Remarks}
We have studied generalised colouring numbers of graphs and proved new upper and lower bounds. These colouring numbers can be used to characterise nowhere dense graph classes, and they are directly related to bounds on the degrees of neighbourhood covers for such classes.

\pagebreak
For graph classes of bounded expansion, the colouring numbers do not
depend on the size of the input graphs, and we may view them as
functions of $r$. For classes of treewidth at most $k$ we have
obtained a polynomial bound, roughly~$r^k$,  for the weak colouring
number, whereas for classes that exclude a fixed graph as a
topological minor we have an exponential bound
$\wcol_r(G)\in\Oof((c_{\CCC}\cdot r)^r)$, where $c_{\CCC}$ is a constant depending only on the class $\CCC$.
Our lower
bound for bounded degree classes shows that we cannot hope to improve
the latter bound substantially. Finally, we negatively answered a question of Joret and Wood, showing
that there exist classes of polynomial expansion which have super-polynomial weak colouring numbers.

\section*{Acknowledgements} We thank Micha\l~Pilipczuk and Felix Reidl for
useful discussions on graph classes of polynomial expansion. 

\bibliographystyle{siamplain}
\bibliography{adm}

\end{document}